\documentclass[final]{algoritmy}
\usepackage{amsmath,amsfonts,amssymb,latexsym,epic,eepic}
\usepackage[dvips]{graphics,epsfig,color}
\usepackage{pstricks}

\newtheorem{remark}{Remark}
\newcommand{\bff}{{\boldsymbol f}}
\newcommand{\bfn}{{\boldsymbol n}}
\newcommand{\bfu}{{\boldsymbol u}}
\newcommand{\bfv}{{\boldsymbol v}}
\newcommand{\bfw}{{\boldsymbol w}}
\newcommand{\bfx}{{\boldsymbol x}}
\newcommand{\bfvphi}{{\boldsymbol \varphi}}
\newcommand{\gradi}{{\boldsymbol \nabla}}
\newcommand{\dive}{{\rm div}}
\newcommand{\xR}{\mathbb{R}}
\newcommand{\xdisc}{{\cal X}_\disc}
\newcommand{\xdiscec}{{\cal X}_{\disc}^{\hspace{.1em}\ec}}
\newcommand{\xLtwo}{{\rm L}^{2}}
\newcommand{\xHone}{{\rm H}^{1}}
\newcommand{\xCzero}{{\rm C}^{0}}
\newcommand{\norm}[2]{\hspace{.2em}|\hspace{-.1em}| #1 |\hspace{-.1em}|_{#2}\hspace{.2em}}
\newcommand{\normsq}[2]{\hspace{.2em}|\hspace{-.1em}| #1 |\hspace{-.1em}|_{#2}^2\hspace{.2em}}
\newcommand{\snorm}[2]{\hspace{.2em}| #1 |_{#2}\hspace{.2em}}
\newcommand{\snormsq}[2]{\hspace{.2em}| #1 |_{#2}^2\hspace{.2em}}
\newcommand{\ec}{{\scalebox{0.6}{$\square$}}}
\newcommand{\ic}{{\scalebox{0.6}{$+$}}}
\newcommand{\edge}{\sigma}
\newcommand{\edges}{{\mathcal E}}
\newcommand{\edgesint}{{\mathcal E}_{{\rm int}}}
\newcommand{\edgesext}{{\mathcal E}_{{\rm ext}}}
\newcommand{\edgesec}{{\mathcal E}_{{\rm int}}^{\hspace{.1em}\ec}}
\newcommand{\edgesic}{{\mathcal E}_{{\rm int}}^{\hspace{.1em}\ic}}
\newcommand{\mesh}{{\mathcal T}}
\newcommand{\NN}{{\mathcal N}}
\newcommand{\GG}{{\mathcal G}}
\newcommand{\disc}{{\raisebox{-0.1em}{\scalebox{0.6}{$\mesh$}}}}
\newcommand{\dx}{\, {\rm d}\bfx}
\newcommand{\dedge}{\, {\rm d}\sigma}
\newcommand{\proj}{r_{\hspace{-0.1em} \disc}}
\newcommand{\ie}{{\em i.e.}}

\title{A class of collocated finite volume schemes for incompressible flow problems}

\author{R. Eymard \thanks{Universit\'e de Marne-la-Vall\'ee, France ({\tt eymard@univ-mlv.fr})}
   \and R. Herbin \thanks{Universit\'e de Provence, France ({\tt herbin@latp.univ-mrs.fr})}
   \and J.-C. Latch\'e
   \and B. Piar \thanks{Institut de Radioprotection et de S\^uret\'e Nucl\'eaire (IRSN), France 
                           ({\tt [jean-claude.latche, bruno.piar]@irsn.fr})}
}

\begin{document}
\AlgLogo{1}{10}
\maketitle

\begin{abstract}
In this paper, we present a class of finite volume schemes for incompressible flow problems.
The unknowns are collocated at the center of the control volumes, and the stability of the schemes is obtained by adding to the mass balance stabilization terms involving the pressure jumps across the edges of the mesh.
\end{abstract}


%
%
%

\pagestyle{myheadings}
\thispagestyle{plain}
\markboth{Eymard {\em et al}}{Collocated FVs for incompressible flows}

\section{Introduction}

The use of collocated finite volumes for fluid flow problems is appealing for several reasons.
Among them, let us mention a very inexpensive assembling step (in particular compared to finite elements, because there is no numerical integration to perform) and the possibility to use, at least to some extent, general unstructured meshes with a low complexity of the data structure (compared with staggered schemes) suitable for the implementation of adaptative mesh refinement strategies.
These features make collocated finite volumes attractive for industrial problems, and they are widely used in Computational Fluid Dynamics, either in proprietary or in commercial (FLUENT, CFX, \dots) codes.

\medskip
However, when applied to incompressible flow problems, cell-centered collocated finite volumes suffer from a lack of coercivity, which is usually handled by a stabilization technique initially proposed by Rhie and Chow \cite{rhi-83-num}, and further developed in subsequent works.
We present here an alternative strategy, based on the addition of "pressure-laplacian-like" stabilization terms in the mass balance (\ie\ the continuity constraint) equation.
In contrast to the case of stabilizations {\it \`a la} Rhie and Chow, for this class of schemes, we are able to prove the stability and convergence of most variants for steady or evolution Stokes or Navier-Stokes equations; optimal (\ie\ first order in energy norms) error bounds are also provided for the Stokes problem \cite{eym-06-sta,eym-07-conv,eym-08-conv}.
Numerical tests for these schemes for a variety of flow problems can be found in \cite{che-06-num,eym-07-new,eym-07-sta, che-08-col}.

\medskip
In this paper, we restrict the exposition to the stationary Stokes problem:
\begin{subequations}
\begin{align} &
-\Delta \bfu + \gradi p = \bff && \hspace{-20ex} \textrm{ in } \Omega, \hspace{20ex}
\label{qdm} \\ &
\dive (\bfu) = 0 \;  && \hspace{-20ex} \textrm{ in } \Omega,
\label{mass} \\ & \displaystyle
\bfu = 0 && \hspace{-20ex} \textrm{ on } \partial \Omega,
\label{BC}\end{align}\label{pbcont_s}\end{subequations}
where $\Omega$ is a polygonal domain of $\xR^2$, $\partial \Omega$ is the boundary of $\Omega$, $\bfu$ stands for the velocity, $p$ for the pressure, the mean value of which is supposed to be zero, and $\bff$ is a forcing term.

\medskip
The stability of the Stokes problem may be readily proved in two steps.
First, multiplying Equation \eqref{qdm} by the unknown $\bfu$ and integrating over $\Omega$ yields:
\[
\int_\Omega |\gradi \bfu|^2 \dx + \int_\Omega \gradi p \cdot \bfu \dx=\int_\Omega \bff \cdot \bfu \dx.
\]
Integrating by parts and using the boundary condition \eqref{BC}, we get:
\begin{equation}
\int_\Omega \gradi p \cdot \bfu \dx=-\int_\Omega p \, \dive \bfu \dx,
\label{div-grad}\end{equation}
so this therm vanishes by \eqref{mass}, and we obtain a control of the velocity $\bfu$ in $\xHone(\Omega)^2$ provided that the forcing term $\bff$ be regular enough, say $\bff \in \xLtwo(\Omega)^2$, which is more than needed in the continuous case but will make the finite volume scheme easier to write.
To obtain a control on the pressure, we use a classical result which is a consequence of a lemma due to Ne\v{c}as:
\begin{equation}
\forall q \in \xLtwo_0(\Omega), \mbox{ there exists } \bfv \in \xHone_0(\Omega)^d \mbox{ such that }
\left| \begin{array}{l}
\dive \bfv = q \mbox{ a.e. in } \Omega,
\\[1ex]
\norm{\bfv}{\xHone(\Omega)^d} \leq c \norm{q}{\xLtwo(\Omega)},
\end{array}\right.
\label{necas}\end{equation}
where $\xLtwo_0(\Omega)$ stands for the subspace of $\xLtwo(\Omega)$ of zero mean value functions and the real number $c$ only depends on $\Omega$.
Choosing $\bfv$ satisfying this relation for $q=p$, multiplying \eqref{qdm} by $\bfv$ and using the estimate for $\bfu$ yields a bound for $\norm{p}{\xLtwo(\Omega)}$.

\medskip
From this computation, we conclude that the stability of the Stokes problem stems from three basic arguments: $(i)$ the coercivity of the diffusion operator, $(ii)$ the duality of the $\gradi$ and $\dive$ operators with respect to the $\xLtwo$ inner-product, $(iii)$ the stability of the gradient operator.
In this paper, we show how to build collocated finite volume schemes satisfying $(i)$ and $(ii)$, and how to circumvent the fact that the property $(iii)$ is not satisfied.
The presentation is organized as follows.
In a first part, we derive the different variants of the proposed schemes for a model problem, namely choosing for the computational domain $\Omega$ the unit square and for the mesh a uniform grid.
In a second part, we briefly discuss how to extend these schemes to general domains and meshes.


\section{A model problem: solving the Stokes system with structured two-dimensional grids}

In this section, we restrict the presentation to the solution of Problem \eqref{pbcont_s} with $\Omega=(0,1)\times(0,1)$ using a structured uniform grid, as sketched on Figure \ref{fig:mesh_and_checkerboard}.
We first describe the discretization, then we present the possible schemes and discuss their stability features.

\subsection{Discrete spaces}

We suppose given a uniform structured mesh $\mesh$ (with step $h$) of $\Omega$, and denote $\edges$ (resp. $\edgesint$, $\edgesext$) the set of edges (resp. internal edges, external edges) of the mesh.
For any two neighbouring control volumes $K$ and $L$ of $\mesh$, we denote by $K|L$ the common edge of $K$ and $L$, and by $\bfn_{K|L}$ the normal vector to $K|L$ oriented from $K$ to $L$ (so $\bfn_{K|L}=-\bfn_{L|K}$).
If $\edge \in \edgesext$, $\bfn_\edge$ stands for the normal vector to $\edge$ outward $\Omega$.
For any control volume $K\in\mesh$, we denote by $\edges(K)$ the set of edges of $K$.
Let $\xdisc \subset \xLtwo(\Omega)$ be the set of functions which are piecewise constant over each control volume.
For any $v \in \xdisc$ and $K \in \mesh$, we denote by $v_K$ the value of $v$ over $K$.
We define a discrete inner product for the functions of $\xdisc$ as follows:
\begin{equation}
\begin{array}{l}
\forall v \in \xdisc,\ \forall w \in \xdisc,
\\[1ex] \displaystyle \hspace{10ex}
(v,w)_\disc=
\hspace{-2ex} \sum_{\begin{array}{c}\scriptstyle \edge \in \edgesint \\[-1ex] \scriptstyle (\edge=K|L) \end{array}} (v_K-v_L)\,(w_K-w_L)
+ \hspace{-2ex} \sum_{\begin{array}{c}\scriptstyle \edge \in \edgesext \\[-1ex] \scriptstyle (\edge \in \edges(K)) \end{array}} 2\, v_K \, w_K.
\end{array}
\label{disc_prod}\end{equation}
This inner product is associated to the norm defined by $\normsq{v}{\disc}=(v,v)_\disc$, $\forall v \in \xdisc$.
This inner product and this norm plays at the discrete level the same role as (and, to some extent, are consistent with) the $\xHone$ inner product and norm in the continuous case; they will be referred to hereafter as the discrete $\xHone$ inner product and norm.
The discrete $\xHone$ norm is known to control the $\xLtwo$ norm \cite{eym-00-fin} (\ie\ there exists a real number $c$ only depending on $\Omega$ and not on the mesh such that the following discrete Poincar\'e relation holds: $\forall v \in \xdisc,\ \norm{v}{\xLtwo(\Omega)}\leq c \norm{v}{\disc}$).
These definitions naturally extend to vector-valued functions by, $\forall \bfv=(\bfv^{(1)},\bfv^{(2)})$ and $\bfw=(\bfw^{(1)},\bfw^{(2)}) \in \xdisc^2$, $(\bfv,\bfw)_\disc=(\bfv^{(1)},\bfw^{(1)})_\disc+(\bfv^{(2)},\bfw^{(2)})_\disc$ and $ \normsq{\bfv}{\disc} = \normsq{\bfv^{(1)}}{\disc} + \normsq{\bfv^{(2)}}{\disc}$.
The following inner product and seminorm will also be used hereafter:
\begin{equation}
\forall p \in \xdisc,\ \forall q \in \xdisc, \quad
[p,q]_\disc=
\hspace{-2ex} \sum_{\begin{array}{c}\scriptstyle \edge \in \edgesint \\[-1ex] \scriptstyle (\edge=K|L) \end{array}} (p_K-p_L)\,(q_K-q_L),
\quad
\snormsq{q}{\disc}=[q,q]_\disc.
\label{disc_snorm}\end{equation}


\subsection{The natural scheme}
Integrating the relations of \eqref{pbcont_s} over each control volume $K$ of the mesh yields:
\[
\begin{array}{l} \displaystyle
\int_{\partial K} -\gradi \bfu \cdot \bfn_{\partial K} \dedge + \int_{\partial K} p\, \bfn_{\partial K} \dedge = \int_K \bff \dx,
\\[2ex] \displaystyle
\int_{\partial K} \bfu \cdot \bfn_{\partial K} \dedge =0,
\end{array}
\]
where $\partial K$ stands for the boundary of $K$ and $\bfn_{\partial K}$ for the normal vector to $\partial K$ outward $K$.
The natural scheme for the solution of Problem \eqref{pbcont_s} thus consists in searching $(\bfu,p)\in \xdisc^2 \times \bar \xdisc$ such that, $\forall K \in \mesh$:
\begin{subequations}
\begin{align} & \displaystyle
(-\Delta_\disc \bfu)_K + (\gradi_\disc p)_K = \bff_K,
\label{nat_qdm} \\[1ex] & \displaystyle
(\dive_\disc \bfu)_K=0,
\label{nat_mass} \end{align}\label{nat_scheme}\end{subequations}
where $\bar \xdisc$ stands for the space of functions of $\xdisc$ with zero mean value, $\bff_K$ is the mean value of $\bff$ over $K$ and
\begin{subequations}
\begin{align} & \displaystyle
(-\Delta_\disc \bfu)_K = \frac 1 {h^2} \sum_{\edge=K|L} (\bfu_K -\bfu_L) + \frac 2 {h^2} \sum_{\edge \in \edges(K)\cap \edgesext} \bfu_K,
\label{disc_lapl} \\ & \displaystyle
(\gradi_\disc p)_K = \frac 1 {h^2} \sum_{\edge=K|L} h\, \frac{p_L+p_K} 2\,\bfn_\edge + \frac 1 {h^2} \sum_{\edge \in \edges(K)\cap \edgesext} h \, p_K \,\bfn_\edge,
\label{disc_grad} \\ & \displaystyle
(\dive_\disc \bfu)_K=\frac 1 {h^2} \sum_{\edge=K|L}  h \, \frac{\bfu_K +\bfu_L} 2 \cdot \bfn_\edge.
\label{disc div} \end{align}\label{disc_ops}\end{subequations}
Since $\forall K \in \mesh$, $\sum_{\edge \in \edges(K)} h \, \bfn_\edge=0$, we have $2\, h\, (\gradi_\disc p)_K = \sum_{\edge=K|L}  (p_L-p_K)\ \bfn_\edge$, and thus, reordering the summations, we get for any pressure $q \in \xdisc$ and velocity $\bfv \in (\xdisc)^2$:
\[
\begin{array}{l} \displaystyle
\int_\Omega \gradi_\disc q \cdot \bfv \dx= \sum_{K\in\mesh} \bfv_K \cdot
\left[ \sum_{\edge=K|L}  h \, \frac{q_L-q_K} 2\,\bfn_\edge \right]=
\\ \hspace{25ex} \displaystyle
-\sum_{K\in\mesh} q_K \left[\sum_{\edge=K|L}  h \, \frac{\bfv_K +\bfv_L} 2 \cdot \bfn_\edge\right]
=-\int_\Omega q\ \dive_\disc \bfv \dx,
\end{array}
\]
which shows that the discrete gradient and divergence operators are transposed operators with respect to the $\xLtwo$ inner product, \ie\ that the stability property $(ii)$ indeed is verified by this scheme.
We then remark that, reordering the summations:
\[
\int_\Omega -\Delta_\disc \bfu \cdot \bfu \dx =\normsq{\bfu}{\disc},
\]
which shows that a discrete equivalent of property $(i)$ is also verified by the scheme.
Mimicking the computation in the continuous case, \ie\ multiplying \eqref{nat_qdm} by $\bfu_K$, reordering the summation, using \eqref{nat_mass} and the discrete Poincar\'e estimate, we thus get a bound for $\bfu$ in the discrete $\xHone$ norm.
To show the stability of the scheme in natural energy norms, the next step would be to control the $\xLtwo$ norm of the pressure through its gradient; unfortunately, the following result shows that it is not possible, at least not uniformly with respect to $h$.

\begin{figure}[tb]
\newgray{graylight}{.97} \newgray{graymid}{.88}
\psset{unit=1cm}
\begin{center}
\begin{pspicture}(0,0)(8.5,4.3)
\rput[bl](0,-0.5){
%
   \pspolygon*[linecolor=graylight](1,1)(2,1)(2,2)(1,2)
   \pspolygon*[linecolor=graymid](2,1)(3,1)(3,2)(2,2)
   \pspolygon*[linecolor=graylight](3,1)(4,1)(4,2)(3,2)
   \pspolygon*[linecolor=graymid](1,2)(2,2)(2,3)(1,3)
   \pspolygon*[linecolor=graylight](2,2)(3,2)(3,3)(2,3)
   \pspolygon*[linecolor=graymid](3,2)(4,2)(4,3)(3,3)
   \pspolygon*[linecolor=graylight](1,3)(2,3)(2,4)(1,4)
   \pspolygon*[linecolor=graymid](2,3)(3,3)(3,4)(2,4)
   \pspolygon*[linecolor=graylight](3,3)(4,3)(4,4)(3,4)
%
   \psline[linecolor=black, linewidth=0.5pt, linestyle=solid]{-}(0.5,1)(4.5,1)
   \psline[linecolor=black, linewidth=0.5pt, linestyle=solid]{-}(0.5,3)(4.5,3)
   \psline[linecolor=black, linewidth=0.5pt, linestyle=solid]{-}(1,0.5)(1,4.5)
   \psline[linecolor=black, linewidth=0.5pt, linestyle=solid]{-}(3,0.5)(3,4.5)
   \psline[linecolor=black, linewidth=0.5pt, linestyle=solid]{-}(0.5,2)(4.5,2)
   \psline[linecolor=black, linewidth=0.5pt, linestyle=solid]{-}(0.5,4)(4.5,4)
   \psline[linecolor=black, linewidth=0.5pt, linestyle=solid]{-}(2,0.5)(2,4.5)
   \psline[linecolor=black, linewidth=0.5pt, linestyle=solid]{-}(4,0.5)(4,4.5)
%
   \psline[linecolor=black, linewidth=0.3pt, linestyle=dashed]{<->}(1,4.6)(2,4.6) \rput[bl](1.4,4.8){$h$}
   \psline[linecolor=black, linewidth=0.3pt, linestyle=dashed]{<->}(0.4,3.)(0.4,4.) \rput[bl](0.,3.3){$h$}
%
   \pspolygon*[linecolor=graylight](5.5,2)(6.,2)(6.,2.5)(5.5,2.5)
   \psline[linecolor=black, linewidth=0.5pt, linestyle=solid]{-}(5.5,2)(6.,2)(6.,2.5)(5.5,2.5)(5.5,2)
   \rput[bl](6.2,2.1){$p_{\rm cb}=1$}
   \pspolygon*[linecolor=graymid](5.5,3)(6.,3)(6.,3.5)(5.5,3.5)
   \psline[linecolor=black, linewidth=0.5pt, linestyle=solid]{-}(5.5,3)(6.,3)(6.,3.5)(5.5,3.5)(5.5,3)
   \rput[bl](6.2,3.1){$p_{\rm cb}=-1$}
}
\end{pspicture}
\end{center}
\caption{Mesh and checkerboard pressure field.}
\label{fig:mesh_and_checkerboard}
\end{figure}
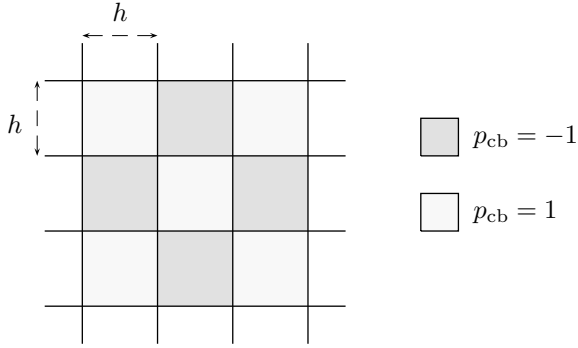

\medskip
\begin{lemma}
We associate to each $K\in\mesh$ its row number $i$ and column number $j$, and define the particular pressure field $p_{\rm cb}$ by $(p_{\rm cb})_K=(-1)^{i+j}$ (see Figure \ref{fig:mesh_and_checkerboard}).
Then the following estimate holds with a real number $c$ independent of $h$:
\[
\forall \bfv \in \xdisc^2, \quad \int_\Omega \gradi_\disc \, p_{\rm cb} \cdot \bfv \dx \leq c\, h \norm{\bfv}{\disc}.
\]
\end{lemma}
\begin{proof}
Since, for any pair of neighbouring control volumes $K$ and $L$, $(p_{\rm cb})_K+(p_{\rm cb})_L=0 $, we have, $\forall K \in \mesh$, $h\, (\gradi_\disc \, p_{\rm cb})_K= \sum_{\edge \in \edges(K)\cap \edgesext} p_K \,\bfn_\edge$.
Thus, $\forall \bfv \in \xdisc^2$:
\[
\int_\Omega \gradi_\disc \, p_{\rm cb} \cdot \bfv \dx =
\sum_{\begin{array}{c}\scriptstyle \edge \in \edgesext \\[-1ex] \scriptstyle (\edge \in \edges(K)) \end{array}}
h\ p_K \,\bfn_\edge \cdot \bfv_K.
\]
By the Cauchy-Schwarz inequality, we obtain:
\[
\int_\Omega \gradi_\disc \, p_{\rm cb} \cdot \bfv \dx \leq
\ \Bigl[
\sum_{\begin{array}{c}\scriptstyle \edge \in \edgesext \\[-1ex] \scriptstyle (\edge \in \edges(K)) \end{array}}
h^2 p_K^2
\Bigr]^{1/2}
\ \Bigl[ 
\sum_{\begin{array}{c}\scriptstyle \edge \in \edgesext \\[-1ex] \scriptstyle (\edge \in \edges(K)) \end{array}}
|\bfv_K|^2
\Bigr]^{1/2},
\]
which concludes the proof since, in the first term, $p_K^2=1$ and so this term is bounded by $4\,h$ and the second one is controlled by $\norm{\bfv}{\disc}$.
\end{proof}


\subsection{A first stabilization}

The basic idea governing the construction of the first stabilized scheme proposed here is to take benefit of the following partial stability result for the discrete gradient.

\begin{lemma}
There exists two positive real numbers $c_1$ and $c_2$ independent of $h$ such that, $\forall q \in \bar \xdisc$, one can find $\bfv \in \xdisc^2$ satisfying:
\[
\norm{\bfv}{\disc}=1 \mbox{ and }
\int_\Omega \gradi_\disc q \cdot \bfv \dx \geq
c_1 \norm{q}{\xLtwo(\Omega)} - c_2 \, h \snorm{q}{\disc}.
\]
\label{stab_grad}\end{lemma}
\begin{proof}
Let $q$ be a function of $\bar \xdisc$.
The idea of this proof is rather natural: let $\tilde \bfv \in \xHone_0(\Omega)^2$ be a function such that \eqref{necas} holds, and let us choose an interpolate of $\tilde \bfv$, say $\bfv \in \xdisc^2$, as test function.
For $\edge \in \edgesint$, $\edge=K|L$, we denote by $\bfv_\edge$ the mean value of $\tilde \bfv$ over $\edge$ and define $\delta \bfv_\edge=(\bfv_K+\bfv_L)/2-\bfv_\edge$.
We suppose that the interpolation operator is stable, in the sense that $\norm{\bfv}{\disc} \leq c \norm{\tilde \bfv}{\xHone(\Omega)^2}$ with $c$ independent of $h$, and is such that $h^2 |\delta \bfv_\edge|^2 \leq c\,h^2 \normsq{\tilde \bfv}{\xHone(K \cup L)^2}$.
Such an interpolation operator is given by instance by simply taking for $\bfv_K$ the mean value of $\tilde \bfv$ over $K$ \cite{eym-08-conv}.
Thanks to the choice of $\tilde \bfv$ and $\delta \bfv_\edge$, we have:
\[
\int_\Omega q\ \dive_\disc \bfv \dx= \normsq{q}{\xLtwo(\Omega)}
+ \sum_{\begin{array}{c}\scriptstyle \edge \in \edgesint \\[-1ex] \scriptstyle (\edge=K|L) \end{array}} h\,(q_K-q_L) \,\delta \bfv_\edge \cdot \bfn_\edge.
\]
By the Cauchy-Schwarz inequality, we get for the last term $T$:
\[
T \leq
\Bigl[
\sum_{\begin{array}{c}\scriptstyle \edge \in \edgesint \\[-1ex] \scriptstyle (\edge=K|L) \end{array}}
(q_K -q_L)^2
\Bigr]^{1/2}
\ \Bigl[
\sum_{\begin{array}{c}\scriptstyle \edge \in \edgesint \\[-1ex] \scriptstyle (\edge=K|L) \end{array}}
h^2 \, |\delta \bfv_\edge|^2
\Bigr]^{1/2}.
\]
The first term is exactly $\snorm{q}{\disc}$ and the second one, thanks to the approximation property of the interpolation operator, is controlled by $h \norm{\tilde \bfv}{\xHone(\Omega)^2}$, itself bounded by $h \norm{q}{\xLtwo(\Omega)}$ thanks to \eqref{necas}.
To conclude the proof, it only remains to normalize $\bfv$ (\ie\ to use $\bfv / \norm{\bfv}{\disc}$ instead of $\bfv$) and invoke the stability of the interpolation operator.
\end{proof}

\medskip
This suggests for a stabilized scheme to search for $(\bfu,p) \in \xdisc^2 \times \bar \xdisc$ such that, $\forall K \in \mesh$:
\begin{subequations}
\begin{align} & \displaystyle
(-\Delta_\disc \bfu)_K + (\gradi_\disc p)_K = \bff_K,
\label{BP_qdm} \\[1ex] & \displaystyle
(\dive_\disc \bfu)_K + \lambda\, h^2 \, (-\Delta_S p)_K=0,
\label{BP_mass} \end{align}\label{BP_scheme}\end{subequations}
with $\lambda>0$ and $\displaystyle (-\Delta_S p)_K=\frac 1 {h^2} \sum_{\edge=K|L} (p_K -p_L)$.

The stabilization term introduced in the mass balance may be seen as a finite volume analogue of the classical so-called Brezzi-Pitk\"aranta stabilization \cite{bre-84-sta} usual in the finite element context.

\begin{theorem}
The scheme \eqref{BP_scheme} admits a unique solution and is stable in natural energy norms, \ie\ there exists a real number $c$ independent of $h$ such that the solution $(\bfu,p)$ of \eqref{BP_scheme} satisfies:
\[
\norm{\bfu}{\disc}+\norm{p}{\xLtwo(\Omega)} \leq c \norm{\bff}{\xLtwo(\Omega)^2}.
\]
\end{theorem}

\begin{proof}
Multiplying \eqref{BP_qdm} by $h^2\,\bfu_K$ and \eqref{BP_mass} by $h^2 \, p_K$ and summing over $K \in \mesh$, we now obtain, thanks to the duality of the discrete gradient and divergence operators:
\[
\normsq{\bfu}{\disc} + \lambda\, h^2 \snormsq{p}{\disc} \leq c.
\]
We now choose $\bfv$ to satisfy Lemma \ref{stab_grad} with $q=p$.
Multiplying \eqref{BP_qdm} by $ h^2 \, \bfv_K$ and summing over $K \in \mesh$, we get:
\[
\norm{p}{\xLtwo(\Omega)} \leq c \Bigl[ h \snorm{p}{\disc} + (\bfu,\bfv)_\disc + \Bigl| \int_\Omega \bff \cdot \bfv \dx \Bigr| \Bigr],
\]
which yields a control on $\norm{p}{\xLtwo(\Omega)}$ and concludes the proof.
\end{proof}


\subsection{A second stabilization and an {\em inf-sup} stability result}

Let us now suppose that the integer number $1/h$ is even.
In this case, the mesh may be partitionned in square $2 \times 2$ patches of control volumes, which are called hereafter "clusters".
The set of internal edges of the mesh $\edgesint$ similarly decomposes into two subsets, $\edgesint=\edgesec \cup \edgesic$, the first one ($\edgesec$) containing the edges separating two control volumes of two different clusters, the second one ($\edgesic$) containing the edges separating two control volumes of a same cluster.
For $q\in \xdisc$,$\snorm{q}{\disc}$ can accordingly be split in two parts, $\snormsq{q}{\disc}=\snormsq{q}{\ec}+\snormsq{q}{\ic}$ with:
\begin{equation}
\snormsq{q}{\ec}=\sum_{\begin{array}{c}\scriptstyle \edge \in \edgesec \\[-0.5ex] \scriptstyle (\edge=K|L) \end{array}} (q_K-q_L)^2
\quad \mbox{and}\quad
\snormsq{q}{\ic}=\sum_{\begin{array}{c}\scriptstyle \edge \in \edgesic \\[-0.5ex] \scriptstyle (\edge=K|L) \end{array}} (q_K-q_L)^2.
\end{equation}
We have the following weak stability result.
\begin{lemma}
There exists a positive real number $c$ independent of $h$ such that, $\forall q \in \xdisc$, one can find $\bfv \in \xdisc^2$ satisfying:
\[
\norm{\bfv}{\disc}=1 \mbox{ and }
\int_\Omega \gradi_\disc q \cdot \bfv \dx \geq c\, h\, \bigl[ \snorm{q}{\ec} - \snorm{q}{\ic} \bigr].
\]
\label{stab_grad_2}\end{lemma}

\begin{figure}[tb]
\newgray{graylight}{.97} \newgray{graymid}{.88}
\psset{unit=1cm}
\begin{center}
\begin{pspicture}(0,0)(8.5,4.3)
\rput[bl](0,-0.5){
%
   \pspolygon*[linecolor=graylight](1,1)(3,1)(3,3)(1,3)
   \pspolygon*[linecolor=graymid](2,2)(3,2)(3,3)(2,3) \rput[bl](2.35,2.4){$K$}
   \rput[bl](1.35,2.4){$L_\ic$} \rput[bl](3.35,2.4){$L_\ec$}
   \rput[bl](2.35,1.4){$M_\ic$} \rput[bl](2.35,3.4){$M_\ec$}
%
   \psline[linecolor=black, linewidth=1pt, linestyle=solid]{-}(0.5,1)(4.5,1)
   \psline[linecolor=black, linewidth=1pt, linestyle=solid]{-}(0.5,3)(4.5,3)
   \psline[linecolor=black, linewidth=1pt, linestyle=solid]{-}(1,0.5)(1,4.5)
   \psline[linecolor=black, linewidth=1pt, linestyle=solid]{-}(3,0.5)(3,4.5)
%
   \psline[linecolor=black, linewidth=0.3pt, linestyle=solid]{-}(0.5,2)(4.5,2)
   \psline[linecolor=black, linewidth=0.3pt, linestyle=solid]{-}(0.5,4)(4.5,4)
   \psline[linecolor=black, linewidth=0.3pt, linestyle=solid]{-}(2,0.5)(2,4.5)
   \psline[linecolor=black, linewidth=0.3pt, linestyle=solid]{-}(4,0.5)(4,4.5)
%
   \psline[linecolor=black, linewidth=1pt, linestyle=solid]{-}(5.5,2)(6.,2)(6.,2.5)(5.5,2.5)(5.5,2)
   \rput[bl](6.1,2.1){cluster}
   \psline[linecolor=black, linewidth=0.3pt, linestyle=solid]{-}(5.5,3)(6.,3)(6.,3.5)(5.5,3.5)(5.5,3)
   \rput[bl](6.1,3.1){control volume}
}
\end{pspicture}
\end{center}
\caption{Clusters and control volumes arrangement and local notations.}
\label{fig:cluster}
\end{figure}
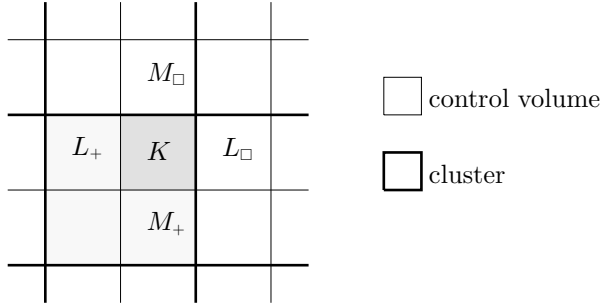

\begin{proof}
Let $q \in \xdisc$ be given, and be such that $\snorm{q}{\ec} \geq \snorm{q}{\ic}$ (otherwise, the result if the lemma is trivial).
Let $K$ be a control volume of $\mesh$ and $L_\ec$, $M_\ec$, $L_\ic$ and $M_\ic$ be its 4 adjacent control volumes, as sketched on figure \ref{fig:cluster}.
We define $\bfv_K$ by:
\[
\bfv_K=\begin{bmatrix} q_{L_\ec} - q_K \\  q_{M_\ec} - q_K \end{bmatrix}.
\]
We have:
\[
\begin{array}{l} \displaystyle
h^2\ (\gradi_\disc q )_K \cdot \bfv_K= \frac h 2\,\Bigl[ (q_{L_\ec} - q_K)^2 + (q_{M_\ec} - q_K)^2
\\ \hspace{20ex}
+ (q_{L_\ec} - q_K)(q_{L_\ic} - q_K)+ (q_{M_\ec} - q_K)(q_{M_\ic} - q_K)\Bigr],
\end{array}
\]
and thus, by Young's inequality:
\[
h^2\ (\gradi_\disc q )_K \cdot \bfv_K \geq  \frac h 4\,\Bigl[ (q_{L_\ec} - q_K)^2 + (q_{M_\ec} - q_K)^2
-(q_{L_\ic} - q_K)^2 - (q_{M_\ic} - q_K)^2 \Bigr].
\]
Summing over the control volumes, we get:
\[
\int_\Omega \gradi_\disc q \cdot \bfv \dx \geq \frac h 2 \, \bigl[ \snormsq{q}{\ec} - \snormsq{q}{\ic} \bigr].
\]
On the other hand, from the expression of $\bfv$, we deduce that $\normsq{\bfv}{\disc} \leq c_3 \snormsq{q}{\disc}$ with $c_3$ independent of $h$, and the conclusion follows by normalizing $\bfv$ and using the fact that, $\forall a,\, b \geq 0,\ a \geq b,\ a^2 - b^2 \geq (a-b) (a^2+b^2)^{1/2}$.
\end{proof}

\medskip
Let us now consider the following scheme, which consists in searching for $(\bfu,p)\in \xdisc^2 \times \bar \xdisc$ such that, $\forall K \in \mesh$:
\begin{subequations}
\begin{align} & \displaystyle
(-\Delta_\disc \bfu)_K + (\gradi_\disc p)_K = \bff_K,
\label{cluster_qdm} \\[1ex] & \displaystyle
(\dive_\disc \bfu)_K + \lambda\, h^2\, (-\Delta_{S,\ic}\, p)_K=0,
\label{cluster_mass} \end{align}\label{cluster_scheme}\end{subequations}
with $\lambda>0$ and $\displaystyle (-\Delta_{S,\ic}\, p)_K=\frac 1 {h^2} \sum_{\edge=K|L,\ \edge\in\edgesic} (p_K -p_L)$.

The stabilization involved in \eqref{cluster_scheme} may be seen as a finite volume analogue of the so-called "local jump stabilization" introduced in \cite{sil-90-sta,kec-92-ana}.

\begin{theorem}
The scheme \eqref{cluster_scheme} admits a unique solution and is stable in natural energy norms, \ie\ there exists a real number $c$ independent of $h$ such that the solution $(\bfu,p)$ of \eqref{cluster_scheme} satisfies:
\[
\norm{\bfu}{\disc}+\norm{p}{\xLtwo(\Omega)} \leq c \norm{\bff}{\xLtwo(\Omega)^2}.
\]
\end{theorem}

\begin{proof}
Multiplying \eqref{cluster_qdm} by $h^2\,\bfu_K$ and \eqref{cluster_mass} by $h^2 \, p_K$ and summing over $K \in \mesh$, we obtain, thanks to the duality of the discrete gradient and divergence operators:
\[
\normsq{\bfu}{\disc} + \lambda \, h^2 \snormsq{p}{\ic} \leq c.
\]
We choose $\bfv$ to satisfy Lemma \ref{stab_grad_2} with $q=p$.
Multiplying \eqref{cluster_qdm} by $h^2\,\bfv_K$ and summing over $K \in \mesh$, we get:
\[
h \snorm{p}{\ec} \leq c \Bigl[ h \snorm{p}{\ic} + (\bfu,\bfv)_\disc + \Bigl| \int_\Omega \bff \cdot \bfv \dx \Bigr| \Bigr],
\]
which gives a control on $h \snorm{p}{\ec}$ and thus on $h \snorm{p}{\disc}$.
The conclusion now follows as for the precedent scheme.
We now choose $\bfv$ to satisfy Lemma \ref{stab_grad} with $q=p$, multiply \eqref{cluster_qdm} by $h^2\,\bfv_K$ and sum over $K \in \mesh$ to get
\[
\norm{p}{\xLtwo(\Omega)} \leq c \Bigl[ h \snormsq{p}{\disc} + (\bfu,\bfv)_\disc + \Bigl| \int_\Omega \bff \cdot \bfv \dx \Bigr| \Bigr],
\]
which yields a control on $\norm{p}{\xLtwo(\Omega)}$.
\end{proof}

Let $\bar \xdiscec \subset \xLtwo_O(\Omega)$ be the space of constant by cluster and zero mean value functions.
Combining lemmata \ref{stab_grad} and \ref{stab_grad_2}, we obtain that, $\forall q \in \bar \xdisc$, $\norm{q}{\xLtwo(\Omega)}$ may be controlled by the gradient of $q$ up to $h \snorm{p}{\ec}$.
Since this latter quantity vanishes for any function of $\bar \xdiscec$, we have the following discrete {\em inf-sup} stability result.

\smallskip
\begin{theorem}
The pair of spaces $\xdisc^2 \times \bar \xdiscec$ is inf-sup stable, in the sense that there exists a positive real number $c$ independent of $h$ such that, $\forall q \in \bar \xdiscec$, there exists $\bfv \in \xdisc^2$ satisfying:
\[
\norm{\bfv}{\disc}=1 \quad \mbox{ and } \quad
\int_\Omega \gradi_\disc q \cdot \bfv \dx \geq c \norm{q}{\xLtwo(\Omega)}.
\]
\label{inf-sup-clust}\end{theorem}

\vspace{-1.5ex}
The pair $\xdisc^2 \times \bar \xdiscec$ thus could be used instead of $\xdisc^2 \times \bar \xdisc$, and the stabilization consequently dropped.
However, from our practice, the second choice is slightly more accurate; it is also easier to implement, since the velocity and the pressure are approximated by the same discrete space.
Note however that recovering a pressure constant by cluster is exactly what happens when the parameter $\lambda$ is large; the accuracy of the scheme then can be expected to be very robust with respect to the value of $\lambda$, which is the main interest of this second stabilization with respect to the first one.
Finally, in the context of transient problems, making use of the {\em inf-sup} stable alternative could be interesting to implement pressure correction schemes, in which stabilizations are difficult to insert.

\begin{remark}[A variational form for this family of schemes]{\rm
The proposed schemes \eqref{BP_scheme} and \eqref{cluster_scheme} may be recast under a "discrete variational form".
For instance, \eqref{BP_scheme} may be written as follows:
\[\begin{array}{ll}\displaystyle
(\bfu,\bfv)_\disc + \int_\Omega p \, \dive_\disc \bfv \dx = \int_\Omega \bff \cdot \bfv \dx,
\quad &
\forall \bfv \in \xdisc^2;
\\[2ex] \displaystyle
\int_\Omega q \, \dive_\disc \bfu \dx + \lambda \, h^2 \, [p,q]_\disc = 0,
&
\forall q \in \xdisc.
\end{array}\]
Indeed, the equations of \eqref{BP_scheme} may be recovered from this formulation by choosing for the test functions the characteristic functions of the control volumes.
This variational formulation is used for the extension of the schemes to more general meshes, in particular by changing the form of the discrete $\xHone$ inner product (see section \ref{sec:suschi}).
}\label{var_form}\end{remark}


\section{Generalizations}

In this section, we turn to the case where $\Omega$ is a polygonal bounded domain of $\xR^2$.
Since the main arguments necessary for the generalization of the schemes described above stem for error estimates, we first address this issue; then two specific cases are treated.


\subsection{Convergence issues}

The error analysis briefly presented here relies of the arguments developed in \cite{eym-00-fin} for the analysis of schemes for elliptic problems.
We consider the scheme for \eqref{pbcont_s} which consists in searching $(\bfu,p)\in \xdisc^2 \times \bar \xdisc$ such that, $\forall K \in \mesh$:
\begin{subequations}
\begin{align} & \displaystyle
(-\Delta_\disc \bfu)_K + (\gradi_\disc p)_K = \bff_K,
\label{gene_qdm} \\[1ex] & \displaystyle
(\dive_\disc \bfu)_K + (T_S)_K=0,
\label{gene_mass} \end{align}\label{gene_scheme}\end{subequations}
where $T_S$ stands for a possible stabilization term, the discrete Laplace operator $\Delta_\disc$ and divergence $\dive_\disc$ read:
\[
(-\Delta_\disc \bfu)_K = \frac 1 {|K|} \sum_{\edge\in\edges(K)} F_\edge(\bfu), \quad
(\dive_\disc \bfu)_K= \frac 1 {|K|} \sum_{\edge=K|L}  G_\edge(\bfu),
\]
and the numerical fluxes $F_\edge(\bfu)$ and $G_\edge(\bfu)$ are functions of the mesh and the value of the unknown $\bfu$ in the control volumes located "near" the edge $\edge$.

We define a set of points $(\bfx_K)_{K \in \mesh}$ such that, for any control volume $K \in \mesh$, the point $\bfx_K$ lies inside $K$.
Then let $\proj$ be the interpolation operator which associates to any function $u \in \xCzero(\Omega)$ the function $\proj u \in \xdisc$ by $\forall K \in \mesh,\ (\proj u)_K = u(\bfx_K)$.
We make the following consistency assumptions:

\medskip
\begin{tabular}{l|l}
$(H_c)$
& \begin{minipage}{0.84\textwidth}
For $\edge\in\edgesint$, the fluxes $F_\edge$ and $G_\edge$ are consistent up to the second order, in the sense that, for any affine vector-valued function polynomial $\bfvphi$:
\[
F_\edge(\proj \bfvphi)=\int_\edge \gradi \bfvphi \cdot \bfn_\edge \dedge, \quad
G_\edge(\proj \bfvphi)=\int_\edge \bfvphi \cdot \bfn_\edge \dedge.
\]
For $\edge\in\edgesext$, $F_\edge$ satisfies the same consistency relation supposing that $\bfvphi$ vanishes on $\edge$ and $G_\edge$ vanishes.
\end{minipage}
\end{tabular}

\medskip
Together with the fact that the scheme is stable in the discrete energy norms, which implies that the assumptions $(i)$ (coercivity of the diffusion term) and $(ii)$ (duality of the discrete gradient and divergence operator) hold, this consistency assumption $(H_c)$ is central for proving first order error estimates (in energy norms) for the Brezzi-Pitk\"{a}ranta stabilization \cite{eym-06-sta} and the stabilization by clusters \cite{eym-08-conv}.


\subsection{Non-uniform structured grids}

We now examine the consequences of these consistency requirements when $\Omega$ is still $(0,1)\times(0,1)$ and the grid is still structured but no-longer uniform.
Let $K$ and $L$ be two adjacent control volumes separated by the edge $\edge$, $h_K^\perp$ (resp. $h_L^\perp$) be the length of $K$ (resp. $L$) in the direction perpendicular to $\edge$.
The natural choice for $\bfx_K$ (resp. $\bfx_L$) is the mass center of $K$ (resp. $L$), and, in this condition, the discretization for $G_\edge$ must be:
\[
G_\edge = |\edge|\ \Bigl[ \frac{h^\perp_L}{h^\perp_K +h^\perp_L}\ \bfu_K + \frac{h^\perp_K}{h^\perp_K +h^\perp_L}\ \bfu_L \Bigr] \cdot \bfn_\edge.
\]
Imposing to the discrete gradient operator to be the transposed of the divergence with respect to the $\xLtwo$ inner product, we obtain that the flux associated to the gradient of the pressure through $\edge$, let say $H_\edge$, reads:
\[
H_\edge=|\edge|\ \Bigl[ \frac{h^\perp_K}{h^\perp_K +h^\perp_L}\ p_K + \frac{h^\perp_L}{h^\perp_K +h^\perp_L}\ p_L \Bigr] \ \bfn_\edge,
\]
which is not the standard (and only a first order) interpolation.

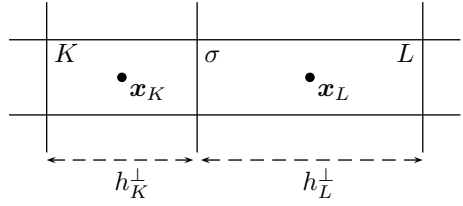
\begin{figure}[tb]
\psset{unit=1cm}
\begin{center}
\begin{pspicture}(0,0)(7,3)
   \psline[linecolor=black, linewidth=0.5pt, linestyle=solid]{-}(0.5,1)(6.5,1)
   \psline[linecolor=black, linewidth=0.5pt, linestyle=solid]{-}(0.5,2)(6.5,2)
   \psline[linecolor=black, linewidth=0.5pt, linestyle=solid]{-}(1,0.5)(1,2.5)
   \psline[linecolor=black, linewidth=0.5pt, linestyle=solid]{-}(3,0.5)(3,2.5)
   \psline[linecolor=black, linewidth=0.5pt, linestyle=solid]{-}(6,0.5)(6,2.5)
   \rput[bl](1.1,1.7){$K$} \rput[bl](5.65,1.7){$L$} \rput[bl](3.1,1.7){$\edge$}
   \psdot*(2,1.5)   \rput[bl](2.1,1.2){$\bfx_K$}
   \psdot*(4.5,1.5) \rput[bl](4.6,1.2){$\bfx_L$}
   \psline[linecolor=black, linewidth=0.5pt, linestyle=dashed]{<->}(1,0.4)(2.95,0.4) \rput[bl](1.9,-0.1){$h^\perp_K$}
   \psline[linecolor=black, linewidth=0.5pt, linestyle=dashed]{<->}(3.05,0.4)(6,0.4) \rput[bl](4.4,-0.1){$h^\perp_L$}

\end{pspicture}
\end{center}
\caption{Notations for an edge in a non-uniform structured grid.}
\label{fig:non_uniform}
\end{figure}


\subsection{General grids}\label{sec:suschi}

A scheme for general grids, including grids involving hanging nodes, is presented in \cite{eym-07-new}.
This scheme may work with the Brezzi-Pitk\"{a}ranta stabilization or with a stabilization by cluster.
For this generalization, two new ingredients, in particular, are necessary:

\smallskip
\begin{itemize}
\item The definition of a diffusion operator.
This is performed using a variational approach with a modified form for the inner product $(\cdot,\cdot)_\disc$, as mentioned in Remark \ref{var_form}.
\item A suitable definition for the clusters, which are seen as patches of elements satisfying the following general condition:
\begin{equation}
\forall K \in \mesh \mbox{ such that } \NN_K \not \subset \GG_K,\quad
\inf_{(a_L) \subset \xR} \frac
{\displaystyle \Bigl| \sum_{L\in \NN_K \setminus \GG_K} a_L\ \bfn_{K|L}\ \Bigr|^2}
{\displaystyle \sum_{L\in \NN_K \setminus \GG_K} a_L^2}
\geq c >0,
\label{cond_cluster}\end{equation}
where, $\forall K \in \mesh$, $\NN_K$ is the set of the neighbours of $K$ (\ie\ the control volumes sharing an edge with $K$) and $\GG_K$ is the cluster containing $K$.
The condition \eqref{cond_cluster} is exactly the condition which allows to control the $\norm{\cdot}{\ec}$ norm of a function of $\xdisc$ by its gradient, as in lemma \ref{stab_grad_2}.
Considering now a family of meshes, this control will be uniform if the real number $c$ does not depend on the considered mesh, and Relation \eqref{cond_cluster} thus acts as a regularity criterion for the meshes.
\end{itemize}

\end{document}